\documentclass{article}
\usepackage{mathtools,amssymb,amsthm}
\usepackage{listings}
\usepackage[dvipsnames]{xcolor}
\usepackage[colorlinks]{hyperref}
\usepackage{authblk}

\newtheorem{theorem}{Theorem}
\newtheorem{lemma}[theorem]{Lemma}

\title{On the Elementary Symmetric Functions of $\{1,1/2,\dots,1/n\}\backslash\{1/i\}$}
\author[1]{Weilin Zhang}
\author[2]{Hongjian Li}
\author[,1,3]{Sunben Chiu\thanks{Corresponding author. Email: zcb@m.scnu.edu.cn.}}
\author[1]{Pingzhi Yuan}
\affil[1]{School of Mathematical Sciences, South China Normal University, Guangzhou, 510631, P. R. China}
\affil[2]{School of Mathematics and Statistics, Guangdong University of Foreign Studies, Guangzhou 510006, P. R. China}
\affil[3]{Department of Technology, ICBC Software Development Center, Guangzhou 510665, P. R. China}
\date{}

\begin{document}
\maketitle
\begin{abstract}
In 1946, P. Erd\H{o}s and I. Niven proved that there are only finitely many positive integers $n$ for which one or more of the elementary symmetric functions of $1,1 / 2$, $\cdots, 1 / n$ are integers. In 2012, Y. Chen and M. Tang proved that if $n \geqslant 4$, then none of the elementary symmetric functions of $1,1 / 2, \cdots, 1 / n$ are integers. In this paper, we prove that if $n \geqslant 5$, then none of the elementary symmetric functions of $\{1,1 / 2, \cdots, 1 / n\} \backslash\{1 / i\}$ are integers except for $n=i=2$ and $n=i=4$.


\end{abstract}

\section{Introduction}

In mathematics, the harmonic series is the infinite series formed by summing all positive unit fractions, that is,
\[
\sum_{n=1}^{\infty} \frac{1}{n}=1+\frac{1}{2}+\frac{1}{3}+\frac{1}{4}+\frac{1}{5}+\cdots,
\]
which is a divergent series and plays an important role in analysis. Adding the first $n$ terms of the harmonic series produces a partial sum, called a harmonic number and denoted $H_n$, that is,
\[
H_n=\sum_{i=1}^n \frac{1}{i} .
\]

In number theory, it is well-known that for any integer $n>1$, the harmonic number $H_n$ is not an integer (see \cite[P33]{7}).

Let $T(n,k)$ denote the $k$-th elementary symmetric function of $1,1 / 2,1 / 3, \cdots,\allowbreak 1 / n$. That is,
\[
T(n,k)=\sum_{1 \leqslant i_1<i_2<\cdots<i_k \leqslant n} \frac{1}{i_1 i_2 \cdots i_k},
\]
where $k$ is an integer with $1 \leqslant k \leqslant n$. With this notation, we mention that $T(n,1)=H_n$. In 1946, P. Erd\H{o}s and I. Niven \cite{2} considered the integrality of $T(n,k)$ and proved that there is only a finite number of positive integers $n$ and $k$ with $1 \leqslant k \leqslant n$ such that $T(n,k)$ is an integer. In 2012, Y. Chen and M. Tang \cite{1} proved that $T(n,k)$ is not an integer except for the following two cases:
\[
T(1,1)=1, T(3,2)=1 \times \frac{1}{2}+1 \times \frac{1}{3}+\frac{1}{2} \times \frac{1}{3}=1 .
\]

For more related problem on this topic, we refer the reader to see \cite{3,4,5,6,8,10,11}.

Let $S(n,i,k)$ denote the $k$-th elementary symmetric function of $\{1,1 / 2, \cdots,\allowbreak 1 / n\} \backslash \{1 / i\}$. That is,
\[
S(n,i,k)=\sum_{\substack{1 \leqslant i_1<i_2<\cdots<i_k \leqslant n, i_j \neq i \text { for } \\ j=1,2, \cdots, k}} \frac{1}{i_1 i_2 \cdots i_k},
\]
where $n$, $i$, $k$ are integers with $1 \leqslant k<n$ and $1 \leqslant i \leqslant n$. Thus, for $n=2$, we have
\[
S(2,1,1)=\frac{1}{2},\quad S(2,2,1)=1 .
\]

For $n=3$, we have
\begin{align*}
S(3,1,1)&=\frac{5}{6}, & S(3,1,2)&=\frac{1}{6}, \\
S(3,2,1)&=\frac{4}{3}, & S(3,2,2)&=\frac{1}{3}, \\
S(3,3,1)&=\frac{3}{2}, & S(3,3,2)&=\frac{1}{2} .
\end{align*}

For $n=4$, we have
\begin{align*}
S(4,1,1)&=\frac{13}{12},&  S(4,1,2)&=\frac{3}{8}, & S(4,1,3)&=\frac{1}{24}, \\
S(4,2,1)&=\frac{19}{12},&  S(4,2,2)&=\frac{2}{3},&  S(4,2,3)&=\frac{1}{12}, \\
S(4,3,1)&=\frac{7}{4}, & S(4,3,2)&=\frac{7}{8},&  S(4,3,3)&=\frac{1}{8}, \\
S(4,4,1)&=\frac{11}{6}, & S(4,4,2)&=1,& S(4,4,3)&=\frac{1}{6} . 
\end{align*}

In this paper, we show that $S(2,2,1)=1$ and $S(4,4,2)=1$ are the only cases for which $S(n,i,k)$ is an integer. We state this result as follows.

\begin{theorem}\label{thm:main}
If $n$, $i$ and $k$ are three positive integers with $n \geqslant 2$, $k<n$ and $i \leqslant n$, then $S(n, i, k)$ is not an integer unless $(n, i, k)=(2,2,1)$ or $(4,4,2)$.
\end{theorem}

The paper is organized as follows. In section \ref{sec:2}, we show several lemmas which are needed for the proof of Theorem \ref{thm:main}. In section \ref{sec:3}, we give the proof of Theorem \ref{thm:main} by using the approach developed in \cite{1,10}.

As usual, we let $\lfloor x\rfloor$ denote the integer part of the real number $x$ and we let $\theta$ be the Chebyshev function. Let $v_p$ denote the $p$-adic valuation on the field $\mathbb Q$ of rational numbers, i.e., $v_p(a)=b$ if $p^b$ divides $a$ and $p^{b+1}$ does not divide $a$.

\section{Technical lemmas}
\label{sec:2}

For our proof of Theorem \ref{thm:main}, we need the following three lemmas.

\begin{lemma}[see {\cite[P359]{9}}]\label{lem:1429}
For $x\ge 1429$, we have
\[
x-\frac{0.334 x}{\ln x}<\theta(x)<x+\frac{0.021 x}{\ln x} .
\]
\end{lemma}

\begin{lemma}\label{lem:max((k+2)(k+3)/2,3k+8)}
Let $k$ and $n$ be positive integers such that $1 \leqslant k<n$. Suppose that there exists a prime $p>\max\{(k+2)(k+3)/2, 3k+8\}$ satisfying that
\[
\frac{n}{k+3}<p \leqslant \frac{n}{k+1} .
\]
Then $S(n, i, k)$ is not an integer for every $i=1,2, \cdots, n$.
\end{lemma}

\begin{proof}
First of all, by the assumption, there is a prime $p>\max\{(k+2)(k+3)/2, 3k+8\}$ such that
\[
\frac{n}{k+3}<p \leqslant \frac{n}{k+1},
\]
that is,
\[
k+1 \leqslant \frac{n}{p}<k+3 .
\]
Thus, $\lfloor n/p\rfloor=k+1$ or $k+2$.

Let $\mathcal{S}_p(n)$ denote the set of all integers from 1 to $n$ that can be divisible by $p$, that is,
\[
\mathcal{S}_p(n)=\bigg\{p, 2 p, 3 p, \cdots,\bigg\lfloor\frac{n}{p}\bigg\rfloor \cdot p\bigg\} .
\]

Now we distinguish two cases as follows.

\textsc{Case 1}: $i \notin \mathcal{S}_p(n)$.

Since $p \nmid i$, we can separate the terms in the sum $S(n, i, k)$ into two parts, depending on whether the denominator is the product of $k$ integers in $\mathcal{S}_p(n)$. That is,
\begin{align*}
S(n, i, k)  &=\sum_{\substack{i_1, i_2, \cdots, i_k \in \mathcal{S}_p(n),\\ i_1<i_2<\cdots<i_k}} \frac{1}{i_1 i_2 \cdots i_k}+\sum_{\substack{1 \leqslant i_1<i_2<\cdots<i_k \leqslant n,\\ \forall i_j \neq i
\text { and } \exists i_t \notin \mathcal{S}_p(n)}} \frac{1}{i_1 i_2 \cdots i_k} \\
 &=\frac{1}{p^k} T\bigg(\bigg\lfloor\frac{n}{p}\bigg\rfloor,k\bigg)+\frac{b}{p^{k-1} c},
\end{align*}
where $b$ and $c$ are two positive integers with $p \nmid c$. Since $p>3k+8$ and
\[
\begin{gathered}
T(k+1, k)=\frac{1}{(k+1) !}\bigg(\sum_{1 \leqslant j \leqslant k+1} j\bigg)=\frac{k+2}{2k !}, \\
T(k+2, k)=\frac{1}{(k+2) !}\bigg(\sum_{1 \leqslant j<s \leqslant k+2} j s\bigg)=\frac{(k+3)(3 k+8)}{24k !} .
\end{gathered}
\]

Hence $v_p(T(\lfloor n/p\rfloor,k))=0$, say
\[
T\bigg(\bigg\lfloor\frac{n}{p}\bigg\rfloor,k\bigg)=\frac{d}{a} \text { for } a, d \in \mathbb{Z} \text { and } p \nmid a d .
\]

Thus
\[
S(n, i, k)=\frac{1}{p^k} T\bigg(\bigg\lfloor\frac{n}{p}\bigg\rfloor,k\bigg)+\frac{b}{p^{k-1} c}=\frac{d c+p a b}{p^k a c}
\]
is not an integer since $v_p(S(n, i, k))=-k<0$. This completes the proof of \textsc{Case 1}.

\textsc{Case 2}: $i \in \mathcal{S}_p(n)$.

In this case, we say $i=p i'$ for some $i' \in\{1,2, \cdots,\lfloor n/p\rfloor\}$. We can also separate the terms in the sum $S(n, i, k)$ into two parts, depending on whether the denominator is the product of $k$ integers in $\mathcal{S}_p(n) \backslash\{i\}$. That is,
\begin{align*}
S(n, i, k)  &=\sum_{\substack{i_1, i_2, \cdots, i_k \in \mathcal{S}_{\mathcal{P}}(n) \backslash\{i\},\\ i_1<i_2<\cdots<i_k }} \frac{1}{i_1 i_2 \cdots i_k}+\sum_{\substack{1 \leqslant i_1<i_2<\cdots<i_k \leqslant n,\\ \forall i_j \neq i \text { and } \exists i_t \notin \mathcal{S}_p(n) \backslash\{i\}}} \frac{1}{i_1 i_2 \cdots i_k} \\
 &=\frac{1}{p^k} S\bigg(\bigg\lfloor\frac{n}{p}\bigg\rfloor,i',k\bigg)+\frac{b}{p^{k-1} c},
\end{align*}
where $b$ and $c$ are two positive integers with $p \nmid c$. Since $p>(k+2)(k+3)/2$ and
\begin{gather*}
S(k+1, i', k)=\frac{1}{(k+1) ! / i'}, \\
S(k+2, i', k)=\frac{1}{(k+2) ! / i'}\bigg(-i'+\sum_{1\leq j\leq k+2} j\bigg)=\frac{(k+2)(k+3)/2-i'}{(k+2) ! / i'} .
\end{gather*}

Hence $v_p(S(\lfloor n/p\rfloor,i',k))=0$, say
\[
S\bigg(\bigg\lfloor\frac{n}{p}\bigg\rfloor,i',k\bigg)=\frac{d}{a} \text { for } a, d \in \mathbb{Z} \text { and } p \nmid a d .
\]

Thus
\[
S(n, i, k)=\frac{1}{p^k} S\bigg(\bigg\lfloor\frac{n}{p}\bigg\rfloor,i',k\bigg)+\frac{b}{p^{k-1} c}=\frac{d c+p a b}{p^k a c}
\]
is not an integer since $v_p(S(n, i, k))=-k<0$. This completes the proof.
\end{proof}

\begin{lemma}\label{lem:2.3}
If $k$ and $n$ are two integers with $n\geqslant 9$ and $n>k \geqslant e \ln n+e$, then $S(n, i, k)$ is not an integer for every $i=1,2, \cdots, n$.
\end{lemma}

\begin{proof}
Proof. By Lemma 3 in \cite{1}, we have $T(n,k)<1$ whenever $e \ln n+e \leqslant k<n$. Note that
\[
T(n,k)=S(n, i, k)+\frac{1}{i} S(n,i,k-1)\text{ for }i=1,2, \cdots, n
\]
Thus, $S(n, i, k)<T(n,k)<1$. It follows that $S(n, i, k)$ is not an integer.
\end{proof}

\section{Proofs}
\label{sec:3}

In this section, we give the proof of Theorem \ref{thm:main}.

\begin{proof}[Proof of Theorem \ref{thm:main}]
By Lemma \ref{lem:2.3}, we may assume that $1 \leqslant k<e \ln n+e$. Now we distinguish three cases as follows.

\textsc{Case 1}: $n\geqslant 50217$.

Since $n\geqslant 50217$, we have $\frac{n}{k+1}>\frac{n}{k+3}>\frac{n}{e \ln n+e+3} \geqslant 1429$. By Lemma \ref{lem:1429}, we get
\begin{align*}
\theta\bigg(\frac{n}{k+1}\bigg)-\theta\bigg(\frac{n}{k+3}\bigg)
& >\frac{n}{k+1}-\frac{0.334 \cdot \frac{n}{k+1}}{\ln \big(\frac{n}{k+1}\big)}-\frac{n}{k+3}-\frac{0.021 \cdot \frac{n}{k+3}}{\ln \big(\frac{n}{k+3}\big)} \\
& >\frac{n}{k+1}\bigg(\frac{2}{k+3}-\frac{0.355}{\ln \big(\frac{n}{k+3}\big)}\bigg) \\
& >\frac{n}{k+1}\bigg(\frac{2}{e \ln n+e+3}-\frac{0.355}{\ln \big(\frac{n}{e \ln n+e+3}\big)}\bigg) \\
& >0. \tag {whenever $n\geqslant 50217$}
\end{align*}
Hence there is a prime $p$ such that $\frac{n}{k+3}<p \leqslant \frac{n}{k+1}$ as desired.

Since $n\geqslant 50217$, it follows that $n>(e \ln n+e+3)(3e \ln n+3e+8)>(k+3)(3k+8)$ and $n>(e \ln n+e+2)(e \ln n+e+3)^2/2>(k+2)(k+3)^2/2$. Thus, we have $p>\frac{n}{k+3}>\max\{(k+2)(k+3)/2, 3k+8\}$. The proof of \textsc{Case} 1 is completed by Lemma \ref{lem:max((k+2)(k+3)/2,3k+8)} immediately.

\textsc{Case 2}: $13543 \leqslant n \leqslant 50216$. After computer verification\footnote{The details can be found at \url{https://github.com/zsben2/esf/tree/main/python}.}, for any $1\leqslant k< e\ln n + e$, there is a prime $p$, such that $\frac{n}{k+3}<p\leqslant \frac{n}{k+1}$ and $p>\max\{(k+2)(k+3)/2, 3k+8\}$.

\textsc{Case 3}: $2 \leqslant n \leqslant 13542$. Denote that $N = 13542$, $K=\lfloor e\ln N + e\rfloor=28$.

Notice that we have the following recursive formulae for $T(n,k)$:
\begin{align*}
T(1,1)&=1;&&\\
T(n,1)&=T(n-1,1)+\frac{1}{n},&& \text{for } 2 \leqslant n\leqslant N;\\
T(n,k)&=T(n-1,k)+\frac{1}{n} T(n-1,k-1), &&\text {for } 2\leqslant k\leqslant K,\ k<n\leqslant N;
\shortintertext{and}
T(n,n)&=\frac{1}{n} T(n-1,n-1),&& \text{for } 2 \leqslant n\leqslant N.
\end{align*}

Also, let $S(1,1,1):=0$, then we have the following recursive formulae for $S(n, i, k)$:
\begin{align*}
S(n,i,1)&=\begin{dcases}
S(n-1,i,1)+\frac{1}{n}, & \text { if } 1\leqslant i<n, \\
T(n-1,1), & \text { if } i=n,
\end{dcases}&&\text{for }2\leqslant n\leqslant N;\\
S(n, i, k)&=T(n,k)-\frac{1}{i} S(n,i,k-1) ,
\tag*{for $2\leqslant n\leqslant N,\ 1\leqslant i\leqslant n,\ 2\leqslant k<\min(n, e \ln n + e)$.}
\end{align*}
In particular, we have
\[
S(n,n,k)=T(n-1,k),\quad\text{for }2\leqslant n\leqslant N,\ 1\leqslant k<\min(n, e \ln n + e).
\]

Using computer\footnote{A sample pseudocode can be found in the appendix \ref{app:Pseudocode}, and the details can be found at \url{https://github.com/zsben2/esf/tree/main/c}.}, we can verify that $S(n, i, k)$ is not an integer for $2 \leqslant n \leqslant 13542$ except for $S(2,2,1) = 1$ and $S(4,4,2) = 1$. This completes the proof of Theorem \ref{thm:main}.
\end{proof}

The experiment is conducted mainly using a Windows 10 (64-bit) system with an Intel (R) Core (TM) i5-8600 CPU @ 3.10 GHz and 16 GB of RAM. The experimental programming language is C (gcc version 8.1.0), using GMP library (version 6.3.0), a GNU Multiple Precision arithmetic library, as the big number calculateion framework. The running time of the experiment under single process is about 540 hours.

\section*{Acknowledgements}

Thanks to Jiayi Zhao from Peking University for renting 4 computers to us, so that we were able to improve the experiment into 6 parallel jobs and run them within 6 days. Pingzhi Yuan is supported by the National Natural Science Foundation of China (Grant No. 12171163).

\bibliographystyle{plain}
\bibliography{refs}

\appendix

\section{Pseudocode}
\label{app:Pseudocode}

First, we have
\begin{equation}
T(n,1) = \frac1n + T(n-1,1),
\end{equation}
for $2 \le n \le N$.

Because there are formulas
\[
T(k,k) = \frac1k T(k-1,k-1),
\]
for $2 \le k \le K$,
\[
T(n,k) = \frac1n T(n-1,k-1) + T(n-1,k),
\]
for $2 \le k \le K$, $k < n \le N$, and when initializing variable \verb|T1|, there is
\[
T(k-1,k) = 0
\]
for $2 \le k \le K$, i.e. $T(1,2) = T(2,3) = \cdots = T(K-1,K) = 0$. So we have
\begin{equation}
T(n,k) = \frac1n T(n-1,k-1) + T(n-1,k)
\end{equation}
for $2 \le k \le K$, $k \le n \le N$.

In addition, the following formulas hold:
\begin{equation}
S(n,i,1) = S(n-1,i,1) + \frac1n
\end{equation}
for $2 <= n <= N$, $1 <= i < n$;
\begin{equation}
S(n,i,k) = T(n,k) - \frac1i S(n,i,k-1) \\
\end{equation}
for $2 \le n \le N$, $1 \le i < n$, $2 \le k < \min\{n, e \ln n + e\}$;
\begin{equation}
S(n,n,k) = T(n-1,k).
\end{equation}
for $2 \le n \le N$, $1 \le k < \min\{n, e \ln n + e\}$.

In the following pseudocode, \verb|r[n]| represents $\frac1n$ for $1\le n \le N$. \verb|S1_1[i]| and \verb|Snik| represent $S(n-1,i,1)$ and $S(n,i,k)$ respectively, for $2 \le n \le N$, $1 \le i \le n$, $1 \le k < \min\{n, e \ln n + e\}$. \verb|T1[k]| and \verb|T2[k]| represent $T(n-1,k)$ and $T(n,k)$ respectively, for $1 \le n \le N$, $1 \le k \le \min(n, K)$.

\begin{lstlisting}[
  language=C,
  basicstyle=\linespread{0.6}\ttfamily\scriptsize,
  keywordstyle=\bfseries\color{NavyBlue},
  commentstyle=\itshape\color{black!50!white},
  numbers=left,
  emph={[1]{zeros,min,print_if_integer,log_e,check_Snik}},
  emphstyle={[1]{\bfseries\color{Rhodamine}}},
  emph={[2]{N,K,e}},
  emphstyle={[2]{\bfseries\color{orange}}},
]
void check_Snik() {
    r = zeros(N + 1); // reciprocal
    S1_1 = zeros(N + 1); // save S(n-1, i, 1)
    T1 = zeros(K + 1); // save T(n-1, k)
    T2 = zeros(K + 1); // save T(n, k)

    r[1] = 1;
    T1[1] = 1;
    for (n = 2; n <= N; n++) {
        r[n] = 1 / n;
        // k = 1
eq1:    T2[1] = r[n] + T1[1];
        // k >= 2
        for (k = 2; k <= min(n, K); k++)
eq2:        T2[k] = T1[k-1] * r[n] + T1[k];

        max_k = (int) min(n - 1, e * log_e(n) + e); // <= K
        // 1 <= i < n
        for (i = 1; i < n; i++) {
            // k = 1
eq3:        Snik = S1_1[i] + r[n];
            print_if_integer(Snik);
            S1_1[i] = Snik; // save S(n, i, 1)

            // 2 <= k <= max_k
            for (k = 2; k <= max_k; k++) {
eq4:            Snik = T2[k] - Snik * r[i];
                print_if_integer(Snik);
            }
        }
        // i = n, 1 <= k <= max_k
        for (k = 1; k <= max_k; k++) {
eq5:        Snik = T1[k];
            print_if_integer(Snik);
        }
        S1_1[n] = T1[1]; // save S(n, n, 1)

        // Rolling to use T1, T2 memory
        T3 = T1;
        T1 = T2;
        T2 = T3;
    }
}
\end{lstlisting}

\end{document}